\documentclass[article,12pt]{amsart}
\usepackage{amsmath,amsthm,amssymb,mathtools,amsfonts,mathrsfs}
\usepackage{xspace,xcolor}
\usepackage[breaklinks,colorlinks,citecolor=teal,linkcolor=teal,urlcolor=teal,pagebackref,hyperindex]{hyperref}
\usepackage[alphabetic]{amsrefs}
\usepackage{tikz-cd}
\usepackage[all]{xy}
\usepackage{bbm}
\usepackage{verbatim}

\newtheorem{thm}{Theorem}[section]
\newtheorem{prop}[thm]{Proposition}

\theoremstyle{definition}

\newtheorem{ex}[thm]{Example}
\newtheorem{rmk}[thm]{Remark}

\newtheorem*{nota*}{Notation}
\newtheorem{ques}{Question}
\newtheorem{lemma}{Lemma}

\newcommand{\on}{\operatorname}

\newcommand{\Aut}{\on{Aut}}

\title[polynomiality of orbifold GW theory of root stacks]{On the polynomiality of orbifold Gromov--Witten theory of root stacks}

\author{Hsian-Hua Tseng}
\thanks{Department of Mathematics, Ohio State University, 100 Math Tower, 231 West 18th Ave. Columbus OH 43210 USA. 
\texttt{hhtseng@math.ohio-state.edu}}
\author{Fenglong You}
\thanks{Department of Mathematical and Statistical Sciences, 632 CAB, University of Alberta, Edmonton AB T6G 2G1 Canada. \texttt{fenglong@ualberta.ca}}

\begin{document}


\maketitle

\begin{abstract}
In \cite{TY18}, higher genus Gromov--Witten invariants of the stack of $r$-th roots of a smooth projective variety $X$ along a smooth divisor $D$ are shown to be polynomials in $r$. In this paper we study the degrees and coefficients of these polynomials.
\end{abstract}
\tableofcontents

\section{Introduction}

\subsection{Overview}
In \cite{TY18}, we showed that, when $r$ is sufficiently large, higher genus orbifold Gromov--Witten invariants of root stacks are polynomials in $r$ and the constant terms are the corresponding relative Gromov--Witten invariants. This result has been generalized to include orbifold invariants of root stacks with large ages by Fan--Wu--You \cite{FWY}, and the constant terms are relative Gromov--Witten invariants with negative contact orders. On the other hand, it was proved in Abramovich--Cadman--Wise \cite{ACW} and Fan--Wu--You \cite{FWY} that genus zero relative and orbifold invariants are equal for sufficiently large $r$. In other words, genus zero orbifold invariants are constant in $r$, that is, they are polynomials in $r$ of degree zero. The genus zero result of \cite{ACW} was reproved in \cite{TY18}*{Section 4} using degeneration and virtual localization techniques. Furthermore, \cite{TY18}*{Theorem 1.9} states that stationary relative and orbifold invariants of target curves coincide in all genera. Hence, stationary orbifold invariants of target curves are also polynomials in $r$ of degree zero. 

In this paper, we answer the following three frequently asked questions related to the polynomiality of orbifold Gromov--Witten invariants of the root stack $X_{D,r}$. 

\begin{ques}\label{ques-1}
What is the bound of the degree of the polynomial?
\end{ques}

\begin{ques}\label{ques-2}
    When is the degree of the polynomial zero? In other words, when will relative and orbifold invariants coincide?
\end{ques}

\begin{ques}\label{ques-3}
    What can be said about other coefficients (coefficients of non-constant terms) of the polynomial?
\end{ques}

\subsection{Relative and orbifold Gromov--Witten invariants}
We review the definition for orbifold Gromov--Witten invariants of root stacks. We refer readers to  \cite{AV}, \cite{AGV02}, \cite{AGV}, \cite{CR} and \cite{Tseng} for the foundation of orbifold Gromov--Witten theory. The construction of root stacks can be found in \cite{AGV}*{Appendix B} and \cite{Cadman}. See also \cite{FWY19}*{Section 2.3} for the definition of orbifold Gromov--Witten invariants of root stacks that we will consider here.

Given a smooth projective variety $X$ with a smooth effective divisor $D$, let $\vec k=(k_1,\ldots,k_m)$ be a vector of $m$ nonzero integers which satisfy 
\[
\sum_{i=1}^m k_i=\int_d[D].
\]
The number of negative elements in $\vec k$ is denoted by $m_-$. 

We assume that $r>|k_i|$ for all $i\in\{1,\ldots, m\}$. We consider the moduli space $\overline{M}_{g,\vec k, n,d}(X_{D,r})$ of $(m+n)$-pointed, genus $g$, degree $d\in H_2(X,\mathbb Q)$, orbifold stable maps to $X_{D,r}$ where 
\begin{itemize}
\item the $i$-th orbifold marking maps to the twisted sector of the inertia stack of $X_{D,r}$ with age $k_i/r$ if $k_i>0$; 
\item the $i$-th orbifold marking maps to the twisted sector of the inertia stack of $X_{D,r}$ with age $(r+k_i)/r$ if $k_i<0$.
\end{itemize}
Let
\begin{itemize}
    \item $\delta_i\in H^*(D,\mathbb Q)$ for $1\leq i \leq m$;
    \item $\gamma_{m+i}\in H^*(X,\mathbb Q)$ for $1\leq i \leq n$
    \item $a_i\in \mathbb Z_{\geq 0}$, for $1\leq i \leq m+n$.
\end{itemize}
Then orbifold Gromov--Witten invariants of $X_{D,r}$ are defined as

\begin{align}\label{orb-inv}
&\left\langle \prod_{i=1}^m\tau_{a_{i}}(\delta_{i})\prod_{i=1}^n \tau_{a_{m+i}}(\gamma_{m+i})\right\rangle^{X_{D,r}}_{g,\vec k,n,d}:=\\
\notag &\int_{[\overline{M}_{g,\vec k,n,d}(X_{D,r})]^{vir}}\bar{\psi}_1^{a_1}\on{ev}^*_{1}(\delta_{1})\cdots \bar{\psi}_m^{a_m}\on{ev}^*_{m}(\delta_{m})\bar{\psi}_{m+1}^{a_{m+1}}\on{ev}^{*}_{m+1}(\gamma_{m+1})\cdots\bar{\psi}_{m+n}^{a_{m+n}}\on{ev}^{*}_{m+n}(\gamma_{m+n}),
\end{align}
where the descendant class $\bar{\psi}_i$ is the class pullback from the corresponding descendant class on the moduli space $\overline{M}_{g,m+n,d}(X)$ of stable maps to $X$. 

Let 
\begin{align}\label{rel-inv}
\left\langle \left.\prod_{i=1}^m\tau_{a_{i}}(\delta_{i})\right|\prod_{i=1}^n \tau_{a_{m+i}}(\gamma_{m+i})\right\rangle^{(X,D)}_{g,\vec k,n,d}
\end{align}
be the corresponding relative Gromov--Witten invariants of $(X,D)$ with contact orders $k_i$ at the $i$-th relative marking, for $1\leq i\leq m$. When $m_-=0$, invariants (\ref{rel-inv}) are simply the relative Gromov--Witten invariants without negative contact orders defined in \cite{LR}, \cite{IP}, \cite{Li1} and \cite{Li2}. When $m_->0$, invariants (\ref{rel-inv}) are the relative Gromov--Witten invariants with negative contact orders defined in \cite{FWY} and \cite{FWY19}. Note that descendant classes in (\ref{rel-inv}) are also defined as classes pullback from the corresponding descendant classes on the moduli space $\overline{M}_{g,m+n,d}(X)$ of stable maps to $X$.

The relation between relative and orbifold Gromov--Witten invariants is as follows.
\begin{thm}[\cite{FWY19}*{Theorem 3.5}]\label{thm:rel-orb}
For $r$ sufficiently large, 
\begin{align}\label{orb-inv-times-m-}
r^{m_-}\left\langle \prod_{i=1}^m\tau_{a_{i}}(\delta_{i})\prod_{i=1}^n \tau_{a_{m+i}}(\gamma_{m+i})\right\rangle^{X_{D,r}}_{g,\vec k,n,d}
\end{align}
is a polynomial in $r$ and the constant term of the polynomial is the corresponding relative Gromov--Witten invariant (\ref{rel-inv}).
\end{thm}

\begin{rmk}\label{rmk:rel-orb}
When $m_-=0$, Theorem \ref{thm:rel-orb} specializes to \cite{TY18}*{Theorem 1.5}. When $g=0$ and $m_->0$, \cite{FWY}*{Theorem 1.1} states that (\ref{orb-inv-times-m-}) is constant in $r$ and equals to the corresponding relative Gromov--Witten invariant (\ref{rel-inv}). When $g=0$ and $m_-=0$, \cite{FWY}*{Theorem 1.1} specializes to \cite{ACW}*{Theorem 1.2.1}.
\end{rmk}

\subsection{Main results}

Question \ref{ques-1}, Question \ref{ques-2} and Question \ref{ques-3} naturally arise from Theorem \ref{thm:rel-orb}. The answer to Question \ref{ques-1} is the following.

\begin{thm}\label{thm:deg}
For $g>0$, the degree of the polynomial 
\[
r^{m_-}\left\langle \prod_{i=1}^m\tau_{a_{i}}(\delta_{i})\prod_{i=1}^n \tau_{a_{m+i}}(\gamma_{m+i})\right\rangle^{X_{D,r}}_{g,\vec k,n,d}
\]
in $r$ is bounded by $2g-1$.
\end{thm}

Question \ref{ques-2} is actually the same as \cite{TY18}*{Question 1.4}. We already know that relative and orbifold invariants are equal in the genus zero case for all smooth projective varieties and in all genera for stationary invariants of target curves. Here, we give a simple criterion for Question \ref{ques-2}.
\begin{thm}\label{thm:rel=orb}
If there are no maps with higher genus components of the source curve mapping into the divisor $D$, then relative and orbifold invariants coincide, for sufficiently large $r$.
\end{thm}

\begin{rmk}
Theorem \ref{thm:rel-orb} can be understood as follows. The moduli space of relative stable maps from smooth curves was compactified by \cite{Li1} and \cite{Gathmann} to define relative Gromov--Witten invariants in the algebraic setting. On the other hand, according to Cadman \cite{Cadman}, the moduli space of orbifold stable maps to the root stack $X_{D,r}$ gives an alternative compactification of the moduli space of relative stable maps with smooth source curves. The discrepancy between relative and orbifold invariants should be a result of different compactifications of the moduli space. Moreover, by \cite{ACW}*{Theorem 1.2.1} and \cite{FWY}*{Theorem 1.1}, genus zero relative and orbifold invariants coincide for sufficiently large $r$. Therefore, the discrepancy in higher genus should be related to stable maps with the higher genus components that map into the divisor $D$.
\end{rmk}

Now we consider Question \ref{ques-3}. By Theorem \ref{thm:rel-orb} and Remark \ref{rmk:rel-orb}, we know that the constant term of the polynomial (\ref{orb-inv-times-m-}) is the corresponding genus $g$ relative invariant and in the genus zero case, the polynomial is just constant. Therefore, it is natural to expect that the coefficients of the non-constant terms are related to relative invariants of lower genera.  
We have the following result.

\begin{thm}\label{thm:coeff}
The coefficients of the non-constant terms of the polynomial 
\[
r^{m_-}\left\langle \prod_{i=1}^m\tau_{a_{i}}(\delta_{i})\prod_{i=1}^n \tau_{a_{m+i}}(\gamma_{m+i})\right\rangle^{X_{D,r}}_{g,\vec k,n,d}
\]
are given by sum over products of lower-genus relative Gromov--Witten invariants of $(X,D)$ and absolute Gromov--Witten invariants of $D$. More specifically, Let $0<j\leq 2g-1$ and $g_0$ be the largest integer such that $j> 2g_0-1$, then the $r^j$-coefficient of the polynomial is determined by relative Gromov--witten invariants of $(X,D)$ of genus less than $(g-g_0)$ and absolute Gromov--Witten invariants of $D$.
\end{thm}

One motivation for studying the polynomiality of orbifold Gromov--Witten theory of root stacks is to use it to better understand relative Gromov--Witten theory. The relation between relative and orbifold Gromov--Witten invariants has led to lots of structural properties of relative Gromov--Witten theory, such as relative quantum cohomology, WDVV equation, topological recursion relation, Givental's formalism and genus zero Virasoro constraint (see \cite{FWY}*{Section 7}).  In \cite{FWY19}*{Section 3.5}, it was shown that relative Gromov--Witten theory is a partial cohomological field theory (partial CohFT) in the sense of \cite{LRZ}, that is,  a CohFT without the loop axiom. On the other hand, orbifold Gromov--Witten theory is a CohFT. The results in our paper may provide a hint for a replacement of the loop axiom, which is also related to Givental's quantization and Virasoro constraint for relative Gromov--Witten theory.

\subsection{Acknowledgement}
We thank Dhruv Ranganathan, Jonathan Wise and Dimitri Zvonkine for important discussions on reduced invariants and the degree of the polynomial. H.-H. T. is supported in part by Simons foundation collaboration grant. F. Y. is supported by a postdoctoral fellowship of NSERC and the Department of Mathematical and Statistical Sciences at the University of Alberta and a postdoctoral fellowship for the Thematic Program on Homological Algebra of Mirror Symmetry at the Fields Institute for Research in Mathematical Sciences.

\section{The degree of the polynomial}

\subsection{Virtual localization formula}

Let $L$ be a line bundle over $D$ and $Y$ be the total space of the $\mathbb P^1$-bundle
\[
\pi: \mathbb P^1(\mathcal O_D\oplus L)\rightarrow D.
\]
The zero and infinity divisors of $Y$ are denoted by $D_0$ and $D_\infty$. We apply the $r$-th root construction to $D_0$ to obtain the root stack $Y_{D_0,r}$. The zero and infinity divisors of $Y_{D_0,r}$ are denoted by $\mathcal D_r$ and $D_\infty$.

Following \cite{TY18}, we can first study the degree of the polynomial for orbifold-relative Gromov--Witten invariants of $(Y_{D_0,r},D_\infty)$. Then the degree of the orbifold Gromov--Witten invariants of $X_{D,r}$ follows from it by the degeneration formula and setting $L=N_D$, the normal bundle to $D$ in $X$. 

We consider the moduli space $\overline{M}_{g,\vec k,n,\vec \mu,d}(Y_{D_0,r},D_\infty)$ of orbifold-relative stable maps with prescribed orbifold and relative conditions given by $\vec k$ and $\vec \mu$ respectively. Let $\epsilon^{\on{orb}}$ be the
forgetful map 
\[
\epsilon^{\on{orb}}:\overline{M}_{g,\vec k,n,\vec \mu,d}(Y_{D_0,r}, D_\infty) \rightarrow \overline{M}_{g,m+n+l(\mu),d}(Y) 
\]
that forgets relative and orbifold conditions, where $l(\mu)$ is the length of $\vec \mu$, that is, the number of elements in $\vec \mu$.

The Gromov--Witten invariants of $(Y_{D_0,r},D_\infty)$ are computed via the virtual localization formula in \cite{JPPZ18}, \cite{TY18} and \cite{FWY19}. The polynomiality follows from the study of the vertex contribution over $D_0$. We refer readers to  \cite{JPPZ18} and \cite{TY18} for details of the virtual localization formula. In particular, we refer to \cite{TY18}*{Section 3.2.1} for the notation of decorated graphs.
\begin{lemma}\label{lemma-loc-formula}
The virtual localization formula can be written as follows;
\begin{align}\label{localization-formula}
[\overline{M}_{g,\vec k,n,\vec \mu,d}(Y_{D_0,r}, D_\infty)]^{\on{vir}}=
\sum_{\Gamma}\frac{1}{|\on{Aut}(\Gamma)|\prod_{e\in E(\Gamma)}d_e} \cdot\iota_*\left(\frac{[\overline{M}_{\Gamma}]^{\on{vir}}}{e(\on{Norm}^{\on{vir}})}\right),
\end{align}
where the sum is taken over decorated graphs $\Gamma$ and $e(\on{Norm}^{\on{vir}})$ is the Euler class of the virtual normal bundle.
\end{lemma}
The inverse of the Euler class of virtual normal bundle can be written as follows:
\begin{itemize}
\item for each stable vertex $v$ over the zero section, there is a factor
\begin{align}\label{loc-contr-0}
&\left(\prod_{e\in E(v)}\frac{rd_e}{t+\on{ev}_{e}^*c_1(L)-d_e\bar{\psi}_{(e,v)}}\right)\cdot\left(\sum_{i=0}^{\infty}(t/r)^{g(v)-1+|E(v)|-i+m_-(v)}c_i(-R^\bullet\pi_*\mathcal L)\right)\\
\notag =& t^{-1}\left(\prod_{e\in E(v)}\frac{d_e}{1+(\on{ev}_{e}^*c_1(L)-d_e\bar{\psi}_{(e,v)})/t}\right)\cdot\left(\sum_{i=0}^{\infty}t^{g(v)-i+m_-(v)}(r)^{i-g(v)+1-m_-(v)}c_i(-R^\bullet\pi_*\mathcal L)\right)\\
\notag=& t^{-1}\left(\prod_{e\in E(v)}\frac{d_e}{1+(\on{ev}_{e}^*c_1(L)-d_e\bar{\psi}_{(e,v)})/t}\right)\cdot\left(\sum_{i=0}^{\infty}(tr)^{g(v)-i+m_-(v)}(r)^{2i-2g(v)+1-2m_-(v)}c_i(-R^\bullet\pi_*\mathcal L)\right),
\end{align}
where 
\[
\pi: \mathcal C_{g(v),\on{val}(v),d(v)}(\mathcal D_r)\rightarrow \overline{M}_{g(v),\on{val}(v),d(v)}(\mathcal D_r)
\]
 is the universal curve, 
\[
\mathcal L\rightarrow \mathcal C_{g(v),\on{val}(v),d(v)}(\mathcal D_r)
\] 
is the universal $r$-th root. 
\item If the target expands over the infinity section, there is a factor
\begin{align}\label{loc-contr-infinity}
\frac{\prod_{e\in E(\Gamma)}d_e}{-t-\psi_\infty}.
\end{align}
\end{itemize}
The contributions in (\ref{loc-contr-0}) and (\ref{loc-contr-infinity}) are all the contributions that appear in the localization formula. We consider the pushforward to $\overline{M}_{g,m+n+l(\mu),d}(Y)$ by $\epsilon^{\on{orb}}$. We define
\[
\hat{c}_i:=r^{2i-2g(v)+1}\epsilon^{\on{orb}}_* c_i(-R^\bullet\pi_*\mathcal L).
\] 
\cite{JPPZ18}*{Corollary 11} states that, for each $i\geq 0$, the class $\hat{c}_i$ is a polynomial in $r$ when $r$ is sufficiently large. Now, we show the following degree bound for the polynomial.

\begin{lemma}\label{lemma-deg-2i}
The degree of the polynomial $\hat{c}_i$ is bounded by $2i$ when $r$ is a sufficiently large integer.
\end{lemma}

We will need the following general property, which we learned from D. Zvonkine, about Ehrhart polynomials. 
\begin{lemma}\label{lem:gen_prop}
Suppose we have an integer polynomial $P(r, a_1, ..., a_N)$ of total degree $d$. In the $N$-dimensional space with coordinates $a_1,..., a_N$ we pick an integral polytope  $\Delta$  of dimension $n$. For $r$ integers, summing the values of $P$ over the integer points of the rescaled polytope  $r\Delta$, then the outcome is a polynomial in  $r$  of degree  $n+d$. 
\end{lemma}
A discussion and further references of Lemma \ref{lem:gen_prop} can be found in \cite[Introduction]{bblkv}.

\begin{proof}[Proof of Lemma \ref{lemma-deg-2i}]
Fix a stable graph $\Gamma$. From the definition it is not hard to see that $r$-twistings form lattice points of a polytope. We need to know its dimension is $h^1(\Gamma)$. This can be seen as follows. Applying Lemma \ref{lem:gen_prop} to the constant polynomial $P=1$, we get the cardinality of the set of $r$-twistings, which according to Lemma \ref{lem:gen_prop} should be a polynomial in $r$ of degree equal to the dimension of the polytope. On the other hand, the cardinality is $r^{h^1(\Gamma)}$ as noted in the bottom of \cite{JPPZ}*{Section 0.4.1}.

Therefore, if $f(tw)$ is an integer polynomial of degree $d$ in the $r$-twisting variables $tw$, then the sum $$\sum_{tw: r-\text{twistings}}f(tw)$$ is a polynomial in $r$ of degree $d+h^1(\Gamma)$. 

In the Grothendieck--Riemann--Roch formula for $c_i$ in \cite{JPPZ18}*{Section 2.4}, the $r$-twisting variables $tw$ appear in the form $(tw)/r$. Also notice that the Grothendieck--Riemann--Roch formula carries a global factor of $r^{2g(v)-1-h^1(\Gamma)}$. Thus after summing over $r$-twistings, we see that the coefficient in $c_i$ corresponding to $\Gamma$ is a Laurent polynomial in $r$ of degree $2g(v)-1$. The result follows. 
\end{proof}

\subsection{Proof of Theorem \ref{thm:rel-orb}}

Then, we have the degree bound for the polynomial of orbifold-relative Gromov--Witten invariants of $(Y_{D_0,r},D_\infty)$.

Given a partion $\vec \mu$ of $\int_d[D]$, a cohomology weighted partion $\mathbf \mu$ is a partion whose parts are weighted by cohomology classes of $H^*(D,\mathbb Q)$. We will use $\mu$ in the correlator notation to specify insertions of relative markings without actually writing out all relative insertions.

\begin{prop}\label{prop-deg-small-age}
We assume that $r$ is a sufficiently large integer. When $m_-=0$, that is, when there are no large-age markings, the degree of the polynomial 
\[
\left\langle\left. \prod_{i=1}^m\tau_{a_{i}}(\delta_{i})\prod_{i=1}^n\tau_{a_{m+i}}(\gamma_{m+i})\right|\mathbf \mu\right\rangle^{(Y_{D_0,r},D_\infty)}_{g,\vec k,n,\vec\mu,d}
\]
in $r$ is $0$ when $g=0$ and bounded by $2g-1$ when $g>0$.
\end{prop}
\begin{proof}
We take the non-equivariant limit of the virtual localization formula. This means that setting $t=0$ in (\ref{localization-formula}). In this case the coefficients of $t^{<0}$ all vanish because non-equivariant limit exists, and the limit is the $t^0$ coefficient. 

First note that the localization graphs for the target geometry $(Y_{D_0,r},D_\infty)$ are in bijection with localization graphs for the target geometry $(Y,D_\infty)$. Hence the number of localization graphs is independent of $r$.

When $g=0$, the result is already proved in \cite{TY18}*{Section 4}. The genus zero orbifold-relative invariants of $(Y_{D_0,r},D_\infty)$ are constant in $r$.

Suppose $g>0$. Expanding the localization contribution, we see that for each localization graph, its contribution to the $t^0$ coefficient is a sum of terms of the form 
\begin{equation}\label{t0}
\prod_v\hat{c}_{i_v}r^{g(v)-i_v}\prod_v\prod_{e_v}d_{e_v}(-\text{ev}_e^*c_1(L)-d_{e_v}\bar{\psi}_{(e_v,v)})^{k_{e_v}}
\end{equation}
where 
\begin{equation}\label{eqn_deg}
\sum_v\sum_{e_v} k_{e_v}=\sum_v(g(v)-i_v-1).
\end{equation}
Since $\sum_{e_v}k_{e_v}\geq 0$, we have $$0\leq \sum_v i_v\leq \sum_v(g(v)-1).$$ 
If these inequalities are impossible (i.e. $\sum_v(g(v)-1)<0$), then this localization graph does not contribute to the $t^0$ coefficient.

By Lemma \ref{lemma-deg-2i}, the $r$-degree of the contribution, which comes from $\hat{c}_{i_v}r^{g(v)-i_v}$, is thus bounded by $$\sum_v(2i_v+g(v)-i_v)=\sum_v(g(v)+i_v)\leq \sum_v(2g(v)-1).$$
This quantity is bounded by $2g-1$, with maximum achieved by the localization graph with one full genus vertex over $0$. 

If the target expands over infinity, (\ref{t0}) contains a factor of the form $\psi_\infty^k$, with (\ref{eqn_deg}) replaced by $1+k+\sum_v\sum_{e_v} k_{e_v}=\sum_v(g(v)-i_v-1)
$. The same argument yields the bound $2g-2<2g-1$.

The result follows.
\end{proof}

To include orbifold invariants with large ages, we need a refined degree bound of Lemma \ref{lemma-deg-large-age} (will be proved later) when $i\geq g(v)$. In Lemma \ref{lemma-deg-large-age}, we will consider the class $$\epsilon^{\on{orb}}_*\left((r)^{i-g(v)+1}c_i(-R^\bullet\pi_*\mathcal L)\right).$$ By \cite{FWY19}*{Corollary 4.2}, this class is a polynomial in $r$ for $r$ sufficiently large. In Lemma \ref{lemma-deg-large-age}, we will prove that the degree of this class is bounded by $2g-1$. To prove the degree bound for this class, we will follow the proof of \cite{FWY19}*{Theorem 4.1} and keeping track of the degree of the polynomial. The basic idea is the following: \cite{FWY19}*{Lemma 4.6} states an equivariant version of \cite{TY18}*{Theorem 2.3} on the cycle level by proving that equivariant theory can be considered as a limit of non-equivariant theory. Hence, the polynomial in \cite{FWY19}*{Lemma 4.6} is of degree $2g(v)-1$. \cite{FWY19}*{Theorem 4.1}, as well as \cite{FWY19}*{Corollary 4.2}, follows from \cite{FWY19}*{Lemma 4.6} by identifying localization residues.

Recall that $\mathcal D_r$ is the zero divisor of the root stack $Y_{D_0,r}$. Let $\overline{M}_{g,\vec a,d}(\mathcal D_r)$ be the moduli space of orbifold stable maps to $\mathcal D_r$, where $\vec a$ is a vector of ages. Let
\[
\pi: \mathcal C_{g,\vec a,d}(\mathcal D_r)\rightarrow \overline{M}_{g,\vec a,d}(\mathcal D_r)
\]
 be the universal curve, 
\[
\mathcal L\rightarrow \mathcal C_{g,\vec a,d}(\mathcal D_r)
\] 
is the universal $r$-th root. We consider the forgetful map 
\[
\epsilon^{\on{orb}}: \overline{M}_{g,\vec a, d}(\mathcal D_r)\rightarrow \overline{M}_{g,l(\vec a),d}(D)
\]
that forgets orbifold conditions.
We would like to show the following degree bound.
\begin{lemma}\label{lemma-deg-large-age}
The degree of the polynomial $$\epsilon^{\on{orb}}_*\left(r^{i-g+1}c_i(-R^\bullet\pi_*\mathcal L)\cap[\overline{M}_{g,\vec a, d}(\mathcal D_r)]^{\on{vir}}\right)$$ is bounded by $2g-1$ for $i\geq 0$.
\end{lemma}

\begin{proof}
First of all, recall that, 
the localization computation of Proposition \ref{prop-deg-small-age} and the degeneration formula imply Theorem \ref{thm:deg} when $m_-=0$. In other words, 
the cycle class
\begin{align}\label{non-equiv-cycle}
\epsilon^{\on{orb}}_*\left[\overline{M}_{g,\vec k,n,d}(Y_{D_0,r})\right]^{\on{vir}}
\end{align}
is a polynomial in $r$ and the degree is bounded by $2g-1$, where we use the same notation $\epsilon^{\on{orb}}$ to denote the forgetful map to $\overline{M}_{g,m+n,d}(Y)$ from different moduli spaces. Note that, this is basically \cite{FWY19}*{Lemma 4.5} with the degree bound and the cycle class (\ref{non-equiv-cycle}) is the orbifold cycle class in \cite{FWY19}*{Lemma 4.5}. 

The next step is to proof the degree bounded for the orbifold cycle class in \cite{FWY19}*{Lemma 4.7}. It can be described as follows. Let $\pi: E\rightarrow B$ be a smooth morphism between two smooth algebraic varieties. Suppose that $E$ is also a $\mathbb C^*$-torsor over $B$. Let 
\[
Y_{D_0,r}\times _{\mathbb C^*}E=(Y_{D_0,r}\times E)/\mathbb C^*
\]
with $\mathbb C^*$ acts on both factors. We consider moduli space $\overline{M}_{g,\vec k,n,d}(Y_{D_0,r}\times_{\mathbb C^*}E)$ of orbifold stable maps to $Y_{D_0,r}\times_{\mathbb C^*}E$, where the curve class $d$ is a fiber class (projects to $0$ on $B$). Let $\left[\overline{M}_{g,\vec k,n,d}(Y_{D_0,r}\times_{\mathbb C^*}E)\right]^{\on{vir}_{\pi}}$ be the virtual cycle relative to the base $B$. Let
\[
\epsilon^{\on{orb}}_{E}: \overline{M}_{g,\vec k,n,d}(Y_{D_0,r}\times_{\mathbb C^*}E) \rightarrow \overline{M}_{g,m+n,d}(Y\times_{\mathbb C^*}E) 
\]
be the forgetful map that forgets orbifold conditions.
Then \cite{FWY19}*{Lemma 4.7} states that the cycle class 
\begin{align}\label{family-cycle}
\left(\epsilon^{\on{orb}}_{E}\right)_*\left[\overline{M}_{g,\vec k,n,d}(Y_{D_0,r}\times_{\mathbb C^*}E)\right]^{\on{vir}_{\pi}}
\end{align}
is a polynomial in $r$. The proof is parallel to the result of (\ref{non-equiv-cycle}) as explained in \cite{FWY19}*{Section 4.2}, hence the degree bound is also $2g-1$ following the same degree computation for (\ref{non-equiv-cycle}) with $m_-=0$ in Proposition \ref{prop-deg-small-age}.

The third step is to proof the same degree bound for the equivariant cycle class
\begin{align}\label{equiv-cycle}
    \epsilon^{\on{orb}}_*\left[\overline{M}_{g,\vec k,n,d}(Y_{D_0,r})\right]^{\on{vir, eq}}.
\end{align}
We follow the proof of \cite{FWY19}*{Section 4.3}. The idea is that equivariant theory can be considered as a limit of non-equivariant theory. By \cite{EG}*{Section 2.2}, the $i$-th Chow group of a space $X$ under an algebraic group $G$ can be defined as follow. Let $V$ be an $l$-dimensional representation of $G$ and $U\subset V$ be an equivariant open set where $G$ acts freely and whose complement has codimension more than $\dim X-i$. Then the $i$-th Chow group is defined as
\begin{align}\label{equiv-chow}
A_i^G(X)=A_{i+l-\dim G}((X\times U)/G).
\end{align}
To apply it to our case, we let $G=\mathbb C^*$ and $E=U=\mathbb C^N-\{0\}$, where $N$ is a sufficiently large integer. Then we have that $(X\times E)/\mathbb C^*$ is an $X$-fibration over $B=U/G=\mathbb P^{N-1}$. Note that
\[
\overline{M}_{g,\vec k,n,d}(Y_{D_0,r}\times_{\mathbb C^*}E)\cong \left(\overline{M}_{g,\vec k,n,d}(Y_{D_0,r})\times E\right)/\mathbb C^*
\]
as moduli spaces. For suitable $N$, (\ref{equiv-cycle}) identifies the equivariant Chow group with a non-equivariant model. Therefore, the equivariant cycle (\ref{equiv-cycle}) is identified with the non-equivariant cycle (\ref{family-cycle}) under (\ref{equiv-chow}). Therefore, the degree bound for (\ref{equiv-cycle}) is also $2g-1$.

The last step is to consider localization residues of $\overline{M}_{g,\vec k,n,d}(Y_{D_0,r})$. We consider the decorated graph with one vertex over $\mathcal D_r$ such that markings and edges are associated with the vector of ages $\vec a$. The localization residue is a polynomial in $r$ and the degree is bounded by $2g-1$. Then the cycle
\[
\epsilon^{\on{orb}}_*\left(\sum_{i=0}^\infty\left(\frac tr\right)^{g-i-1}c_i(-R^\bullet\pi_*\mathcal L)\cap[\overline{M}_{g,\vec a, d}(\mathcal D_r)]^{\on{vir}}\right),
\]
coming from the localization residue as computed in Lemma \ref{lemma-loc-formula} but with $m_-=0$, is a polynomial in $r$ and the degree is bounded by $2g-1$.
This is the conclusion of \cite{FWY19}*{Theorem 4.1} with the degree bound for the polynomial. As a consequence (see also \cite{FWY19}*{Corollary 4.2}), the cycle
\[
\epsilon^{\on{orb}}_*\left((r)^{i-g+1}c_i(-R^\bullet\pi_*\mathcal L)\cap[\overline{M}_{g,\vec a, d}(\mathcal D_r)]^{\on{vir}}\right)
\]
is a polynomial in $r$ and the degree is bounded by $2g-1$. This concludes the lemma.
\end{proof}

Note that, when $i\geq g$, Lemma \ref{lemma-deg-large-age} provides an improved degree bound for the cycle comparing to the degree bound in Lemma \ref{lemma-deg-2i}.

Now we are ready to consider orbifold invariants with large ages.
\begin{prop}\label{prop:deg}
The degree of the polynomial 
\[
r^{m_-}\left\langle\left. \prod_{i=1}^m\tau_{a_{i}}(\delta_{i})\prod_{i=1}^n\tau_{a_{m+i}}(\gamma_{m+i})\right|\mathbf \mu\right\rangle^{(Y_{D_0,r},D_\infty)}_{g,\vec k,n,\vec\mu,d}
\]
in $r$ is $0$ when $g=0$ and bounded by $2g-1$ if $g>0$.
\end{prop}

\begin{proof}
The proof is similar to the proof of Proposition \ref{prop-deg-small-age} after using Lemma \ref{lemma-deg-large-age}. We need to prove Proposition \ref{prop-deg-small-age} first because we need to use it to prove Lemma \ref{lemma-deg-large-age}.

When $g=0$, the result is already known in \cite{FWY}*{Theorem 1.1}. 

When $g>0$, following the proof of Proposition \ref{prop-deg-small-age}, we take the $t^0$-coefficient of the localization contributions. Applying the degree bound of Lemma \ref{lemma-deg-large-age} to the $t^0$-coefficient of the localization contributions (after multiplying by $r^{m_-}$), we see that the localization contribution from each vertex $v$ is a polynomial in $r$ and the degree is bounded by $2g(v)-1$. 
Therefore, the degree bound for the total contribution is $2g-1$. This concludes the proposition.
\end{proof}

\begin{proof}[Proof of Theorem \ref{thm:deg}]
It follows from Proposition \ref{prop:deg} by the degeneration formula applied to the degeneration of $X_{D,r}$ into $X\cup_D Y_{D_0, r}$. 
\end{proof}

\begin{ex}
When $g=1$, then $2g-1=1$. So the genus one orbifold invariants is a linear function in $r$. We can revisit the counterexample in \cite{ACW}*{Section 1.7}. Let $X=E\times \mathbb P^1$, where $E$ is an elliptic curve. Let $D=X_0\cup X_\infty$ be the union of $0$ and $\infty$ fibers over $\mathbb P^1$. Taking a fiber class, then the genus one orbifold invariants with no insertions is a linear function of roots $r$ and $s$.  
\end{ex}

\subsection{Proof of Theorem \ref{thm:rel=orb}}
Now we turn to Theorem \ref{thm:rel=orb}.
\begin{proof}[Proof of Theorem \ref{thm:rel=orb}]
Localization computation also shows that relative and orbifold invariants are the same if there are no stable maps with higher genus components map into the divisor $D_0$. In this case, the vertex contribution has only genus zero contributions (that is, $g(v)=0$ for all vertex $v$ over zero) and the computation becomes similar to the genus zero case in \cite{TY18}*{Section 4} and \cite{FWY}*{Section 6}.
\end{proof}

\begin{rmk}
Jonathan Wise told us that, according to Dhruv Ranganathan, it is expected that reduced relative invariants and reduced orbifold invariants are equal. While the foundation for these invariants has not fully set-up (see, for example, \cite{BNR}), our computation suggests that the equality should hold for reduced invariants.
\end{rmk}

\section{The coefficients of the polynomial}
In this section, we prove Theorem \ref{thm:coeff} by considering the degeneration of $X_{D,r}$ into $X\cup_D Y_{D_0, r}$.  
By \cite{AF}*{Theorem 0.4.1}, the degeneration formula for Gromov--Witten invariants of $X_{D,r}$ is the following:

\begin{align}\label{degeneration-orbifold}
&\left\langle \prod_{i=1}^m\tau_{a_{i}}(\delta_{i})\prod_{i=1}^n \tau_{a_{m+i}}(\gamma_{m+i})\right\rangle^{X_{D,r}}_{g,\vec k,n,d}=\\
\notag&\sum  \frac{\prod_i \eta_i}{|\on{Aut}(\eta)|} \left\langle\left.\prod_{i=1}^m\tau_{a_{i}}(\delta_{i})\prod_{i\in S}\tau_{a_{m+i}}(\gamma_{m+i})\right |\eta\right\rangle^{\bullet,(Y_{D_0,r},D_\infty)}_{g_1,\vec k,|S|,\vec \eta,d_1}\left\langle\eta^\vee\left|\prod_{i\not\in S}\tau_{a_{m+i}}(\gamma_{m+i})\right.\right\rangle^{\bullet,(X,D)}_{g_2,\vec \eta,n-|S|,d_2}, 
\end{align}
where $\eta^\vee$ is defined by taking the Poincar\'e duals of the cohomology weights of the cohomology weighted partition $\eta$; $|\Aut(\eta)|$ is the order of the automorphism group $\Aut(\eta)$ preserving equal parts of the cohomology weighted partition $\eta$. The sum is over all splittings of $g$ and $d$, all choices of $S\subset \{ 1,\ldots,n \}$, and all intermediate cohomology weighted partitions $\eta$ such that invariants on the right-hand side of (\ref{degeneration-orbifold}) do not vanish. The superscript $\bullet$ stands for possibly disconnected Gromov-Witten invariants. Note that
\[
g_1+g_2+l(\eta)-1=g.
\]

\begin{proof}[Proof of Theorem \ref{thm:coeff}]
Assume $g>0$. We consider 
\[
r^{m_-}\left\langle \prod_{i=1}^m\tau_{a_{i}}(\delta_{i})\prod_{i=1}^n \tau_{a_{m+i}}(\gamma_{m+i})\right\rangle^{X_{D,r}}_{g,\vec k,n,d},
\]
where $m_-$ is the number of negative elements in the partition $\{k_i\}_{i=1}^m$.
For each summand in (\ref{degeneration-orbifold}), if $g_2=g$, then $g_1=0$. In this case, the summand is constant in $r$ by \cite{FWY}*{Theorem 6.1} as genus zero invariants of $(Y_{D_0,r},D_\infty)$ (multiplied by $r^{m_-}$) are constant in $r$. When $m_-=0$, it follows from \cite{ACW}*{Theorem 1.2.1}, see also \cite{TY18}*{Theorem 4.1}. Therefore, the non-constant terms in $r$ appears when $g_1>0$ and $g_2<g$.

Let $j>0$. We consider the $r^j$-coefficient. By Theorem \ref{thm:deg} we can assume that $j\leq 2g-1$. Let $g_0$ be the largest integer such that $j> 2g_0-1$. 
In the degeneration formula (\ref{degeneration-orbifold}), if $g_1\leq g_0$, then by Proposition \ref{prop:deg}, the first factor in the summand is of $r$-degree bounded by $2g_1-1\leq 2g_0-1< j$. Thus terms with $g_1\leq g_0$ do not contribute to the $r^j$ coefficient, and $r^j$ coefficient is formed by the following quantities
\begin{enumerate}
    \item 
    genus $g_1$ Gromov--Witten invariants of $Y_{D_0,r}$ with $g_1> g_0$.
    \item
    relative Gromov--Witten invariants of $(X,D)$ of genus $g_2< g-g_0$. Note that $g-g_0\leq g$.
\end{enumerate}
By localization formula, Gromov--Witten invariants of $Y_{D_0,r}$ are expressed in terms of twisted Gromov-Witten invariants of $D$ and of root gerbe $D_r$. By quantum Riemann--Roch theorems of Coates--Givental \cite{CG} and Tseng \cite{Tseng}, these twisted invariants are expressed in terms of untwisted invariants of $D$ and $D_r$. By gerbe duality results proven by Andreini--Jiang--Tseng \cite{AJT} and Tang--Tseng \cite{TT}, invariants of $D_r$ are determined by invariants of $D$. 
\end{proof}


\bibliographystyle{amsxport}
\bibliography{universalbib}

\end{document}